\RequirePackage{fix-cm}

\documentclass[a4paper]{article}
\usepackage{amsmath, amssymb, amsfonts, amsthm}
\usepackage{graphicx, array}

\newcommand{\sint}{\mathop{\mathrm{int}}\nolimits}

\newcolumntype{C}{>{\centering\arraybackslash} m{.16\linewidth} } 
\newcolumntype{L}{>{\centering\arraybackslash} m{.25\linewidth} } 

\newtheorem{theorem}{Theorem}
\newtheorem{corollary}{Corollary}
\newtheorem{lemma}{Lemma}
\newtheorem{definition}{Definition}

\begin{document}

\title{Inverse Graphical Method for Global Optimization and
    Application to Design Centering Problem}
\author{Sergey Karpukhin\thanks{email: contact@kserz.rocks}
\thanks{Independent researcher,
graduated from the Lomonosov Moscow State University.}
}
\maketitle

\begin{abstract}
    Consider the problem of finding an optimal value of some objective
    functional subject to constraints over numerical domain. This type
    of problem arises frequently in practical engineering tasks.
    Nowdays almost all general methods for solving such a problem are based on
    user-supplied routines computing the objective value at some points.
    We study another approach called inverse relying on some procedure
    to estimate the set of points instead having objective values bounded by
    a specified constant.

    In particular, we present an inverse optimization algorithm derived from
    the bisection of the objective range.
    In case of seeking a proven global optimal solution
    inherently requiring many computations, and a problem with some kind of
    coherency utilized in estimation procedure, the inverse scheme is much
    more efficient than conventional ones. An example of such a problem, namely
    the design centering, is studied to compare the approaches.
\end{abstract}

\section{Introduction}
The numerical global optimization problems are of big importance for science,
economics and technology. They are studied extensively for decades
\cite{Horst1995}, \cite{Nocedal2006}, \cite{Scholz2012}, \cite{Tuy2016}.
We focus specifically on deterministic approaches, i.e. those leading to a
proven result inevitably. The problem is usually formulated as follows:
\begin{equation}
    \begin{gathered}
        f(x) \to \min \\
        \mbox{s.t.}\; x\in F \subset X,
    \end{gathered}
\end{equation}
where $f\colon X\to\mathbb{R}$ is an objective functional,
$X$ is a problem domain and $F$ is some
part of the domain called \emph{feasible region}.
The problem is \emph{constrained} if $F \ne X$ and \emph{unconstrained}
otherwise. Restriction $x\in F$ is often
formulated in terms of some equality or inequality constraints, but this is
not necessarily. Note that in
case of constant $f(x)$ the problem is called \emph{constraint satisfaction}
and there are plenty of methods solving this kind of problem.

According to the literature,
it is almost always assumed in general optimization methods
that we have some means to compute or bound $f(x)$ over specified subset
$X'\subset X$, and the solution algorithms are then based on such
routines \cite{Scholz2012}, \cite{Tuy2016}, \cite{Sergeyev2015}.
Custom algorithms are usually
developed for particular explicit forms of objective and constraints
\cite{Horst1995}, \cite{Neumaier2004}. 
We hereafter call this conventional scheme \textbf{straight}.
In this work we present an opposite
approach to global optimization called \textbf{inverse}.
It is based on computing domain subsets for specified objective bounds, see
Fig.~\ref{inverse_figure}.

\begin{figure}[h]
    \begin{center}
        \includegraphics[width=0.8\textwidth]{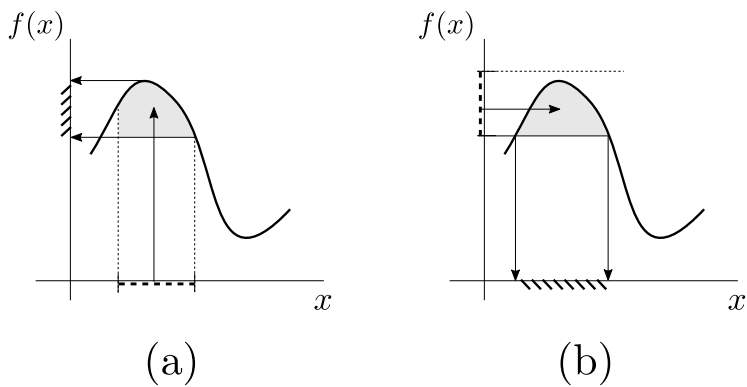}
    \end{center}
    \caption{Computations within optimization: (a) straight(conventional),
    (b) inverse}
    \label{inverse_figure}
\end{figure}

Such framework is rather common in constraint satisfaction problems, but to
the best of our knowledge no similar practical methods have been developed for
optimization problems. It is worth noting that
the performance of the new approach highly depends on a good
domain estimation procedure provided for a particular problem. We
present an example
problem at the end of the article demonstrating a superiority of the inverse
scheme while looking for a proven global optimal solution.

The rest of the article is organized as follows. We start with a brief review
of the related research. Then we proceed to development
of the inverse graphical method and proof of convergence.
The second part of this work deals with a design centering problem
arising in the diamond cutting industry. We discuss implementation details
and compare the inverse approach to the state of
the art technique as well as to the publicly available software libraries.

\section{Related work}
The closest optimization scheme so far has been pioneered
by Sergey~P. Shary \cite{Shary2001} and studied in some works afterwards
\cite{Abaffy2018}, \cite{Contreras2016}. We advance this study by switching
entirely to inverse computations as shown in Fig.~\ref{inverse_figure}.
This is done theoretically by stating sufficient convergency conditions
and practically by developing complete algorithm for a specific problem.
Note that although the works mentioned above deal with the algorthims in terms
of equation solution, which can be thought of as a kind of inverse technique,
the latter is still reduced to straight interval computations there. In practice
this often means utilizing rather complex and expensive procedures. By contrast,
the inverse scheme based on domain estimations could be implemented much simpler
and faster for particular problems.
Besides, we involve an inherently approximate computation framework
providing more freedom to choose effective algorithms.

In addition to the above, we must note the following closely related research.
The main idea of computing the domain region
satisfying some restrictions is very common in constraint satisfaction study
\cite{Rossi2006}. One notable example is the
method based on the feasible region approximation with the help of spatial
trees \cite{SamHaroud1996}. The inverse graphical
method presented later in this article
adapts similar ideas to the
optimization problem. It can be thought of as a sequence of constraint
satisfaction problems with added parametrized optimization constraint
\cite{Emden2003}, \cite{Rossi2006}.
More ideas on interchange between global optimization and
constraint satisfaction can be found in \cite{Hooker2006}, \cite{Hooker2017}
and references therein.

Besides, the inverse optimization method
can be thought of as a sequence of approximate
problem solutions parametrized by the accuracy, or equivalently by the
bounds on acceptable objective value. It is tightly connected to the successive
approximation global optimization methods \cite{Horst1996} as well as
to the relief indicator method \cite{Thach1990}. Compared to them the
top-level algorithm scheme and the theoretical results are alike, but
there are two differences of the new method. First, we connect tightly
particular domain approximations and objective value boundaries.
Precisely, we choose numerical approximations of the problem instead
of adding
new analytic constraints, and estimate feasible regions with the same tolerance
as objective in contrast to accurate solving.
Secondly, we involve easier graphical domain estimation technique
at every iteration instead of constructing complex relaxations and
underestimations \cite{Tuy2016}. These aspects together
allow us to simplify the coarse steps of the overall algorithm and achieve a
performance boost.

Furthermore, we can note non-complete optimization
procedures called simulated annealing and threshold accepting
\cite{Dueck1990}. They reduce the accuracy value as a
parameter much like the inverse approach.
However, these methods are non-deterministic, thus differ entirely from the
subject of this article.

Now let us describe briefly the relations between the inverse scheme and
conventional global optimization techniques.
First, the inverse approach is in fact a numerical search of the
optimal objective value. It is close to the well-known bisection
method and its derivatives \cite{Wood1992}, but we bisect only on the objective
range and use the domain estimation as a decision value.
Comparing to the common branch-and-bound technique \cite{Scholz2012}
the difference is the same:
we subdivide("branch") the objective value range and
estimate("bound") the domain instead of
subdividing the domain and estimating the objective. 
Secondly, our
approach can be treated as the opposite to the penalty function method
\cite{Nocedal2006} incorporating the constraints into the objective
function.
In contrast, we incorporate the bounded objective value as a new
constraint and search for feasible points at every iteration of
the inverse method.

At last, all the references concerning the design centering problem and
the practical application of the inverse scheme are discussed in the second
part of the article.

\section{General inverse graphical optimization method}
\label{gen_section}
This section deals with the basic study of the inverse
method. We will first proceed through the algorithm and then
move to the proof of its
convergence.

We derive the inverse graphical optimization method from the bisection methods
\cite{Shary2001},\cite{Wood1992},\cite{Abaffy2018}. In particular, we split
the infinite objective range by some threshold value and then eliminate a
part of it at every step of the algorithm. The elimination is done through the
estimation of the domain region having objective values within the given range.
With the objective value restricted to a range
the problem transforms to the constraint
satisfaction one and is solved by approximate procedures similarly to
\cite{SamHaroud1996}. The global optimal solution is then acquired reducing
the objective range while ensuring that the corresponding feasible region
is not empty. Refer to Fig.~\ref{inverse_scheme_figure} for general scheme.

\begin{figure}[h]
    \begin{center}
        \includegraphics[width=0.9\textwidth]{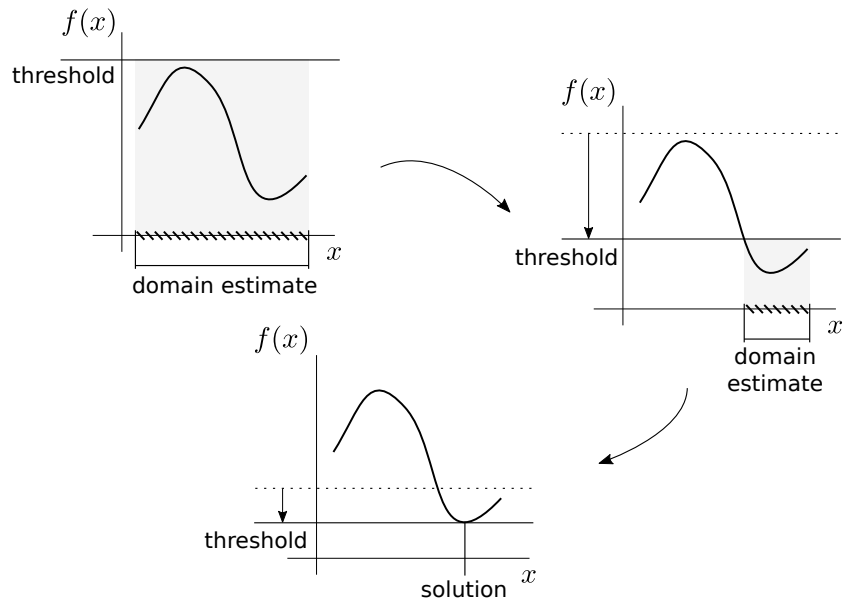}
    \end{center}
    \caption{General scheme of the inverse optimization technique}
    \label{inverse_scheme_figure}
\end{figure}

Note that the core part of the inverse method is a domain estimation
procedure applied at every step. One could imagine many possible variants,
but we describe a specific one leaving the rest to the further research.
Precisely, we approximate the domain by equally-sized hyperrectangles and
then mark them according to the objective values. This approach is very
close to the raster computer graphics. Moreover, we use some rasterization
techniques in the application. Thus we call our variant of the
inverse method \textbf{graphical}.

Let us formalize the problem to solve. In this article we consider the
following optimization problem:
\begin{equation}
    \begin{gathered}
        f(x) \to \min \\
        \mbox{s.t.}\; x\in F \subset X
    \end{gathered}
    \label{problem}
\end{equation}
where $f\colon X \to \mathbb{R}$ is a continuous function,
$X = [x_1^L, x_1^R]\times[x_2^L, x_2^R]\times\cdots\times[x_n^L, x_n^R]
\subset \mathbb{R}^n$ is a compact hyperrectangle in Euclidean space
and
$F\subset X$ is a closed region called the \textbf{feasible region}.
As a restriction of our study we additionally assume that the feasible region
contains an interior point, this also leads to existence of the solution.
Note that in practice one could take the feasible region bounding box as $X$.

In the rest of the article by $\mathcal{B}$ we denote the set of all
hyperrectangles in $\mathbb{R}^n$ and any $B\in\mathcal{B}$ we call simply a
\textbf{box}. Note that all the boxes are closed sets.
For any $B\in\mathcal{B}$ by $\delta(B)$ we denote the diameter
of $B$ as a set, i.e. the maximum distance between points in $B$. For
any set $S\subset\mathbb{R}^n$ by $\sint(S)$ we denote the interior
of $S$ and by $\partial S$ we denote the boundary of $S$.

\subsection{The algorithm}
\label{algorithm_subsection}
Now we develop the inverse graphical optimization algorithm. Consider a
problem of form (\ref{problem}) and a \textbf{target accuracy
$\varepsilon > 0$} for the solution.
The key feature of the algorithm is the assumption that
we have a \textbf{domain estimation procedure} for a particular problem.
By the domain estimation procedure we consider some algorithm to compute
a relation
\[
    C\colon \mathcal{B}\times \mathbb{R} \to \{0, \beta, 1\}
\]
such that for
any box $B\in\mathcal{B}$ and threshold $t\in\mathbb{R}$
\[
\begin{aligned}
    C(B,t) &= 0 \mbox{ only if for all $x\in B$ $x \not\in F$ or $f(x) > t$},\\
    C(B,t) &= 1 \mbox{ only if for all $x\in B$ $x \in F$ and $f(x) \le t$},\\
    C(B,t) &= \beta \mbox{ in any other or unknown case}.
\end{aligned}
\]
According to this for a given threshold $t$ we call the box $B$
\textbf{empty} if $C(B,t) = 0$,
\textbf{filled} if $C(B,t) = 1$ and \textbf{boundary} if $C(B,t) = \beta$.
For more insights on such a
procedure, see \cite{SamHaroud1996}: our empty, boundary and filled boxes
roughly correspond to the black, grey and white ones in \cite{SamHaroud1996}.
Some ideas on constructing the practical domain estimation procedure can be
found in \cite{SamHaroud1995} and we provide some research for a different
problem in the second part of the article. Note that any filled box lies
completely inside the feasible region.
Moreover, we emphasize that any actually empty or filled box in general
can be labeled as  boundary by our definition of the procedure, which is not the
case in \cite{SamHaroud1995},\cite{SamHaroud1996}.

Intuitively these notions represent the fact that when we seek for solutions
to (\ref{problem}) with the objective value of $t$ or better.
Thus, box $B$ is empty if there are no such points in it and $B$ is filled if
any point of $B$ constitute such a solution. The boundary boxes represent the
area of uncertainty of either the problem or the algorithm. The uncertainties
of the problem are the boxes with objective values on both sides
of threshold $t$ or both feasible and infeasible.
The algorithm uncertainties may come from the approximating techniques,
for example when precise domain estimation needs a lot of computational
resources and we consider some rough but faster algorithm. We state
some sufficient conditions on the choice of the algorithms to ensure the
convergence of the method.

To continue developing, assume a given parameter $S$ of the algorithm called
\textbf{domain scale}. This parameter controls the balance between domain 
regions sizes and objective values at iteration steps.
It seems that $S$ should be connected
to the Lipschitz constant of objective, but this question have not been
studied yet. As a rule of thumb, $S=0.6$ performed well enough for the design
centering problems tested in the second part of the work. The theoretical
results do not depend on the value of $S$ as well.

With all the definitions above the inverse graphical optimization
algorithm is defined as follows.

\noindent
\begin{minipage}{\textwidth}
\medskip
\hrule
\smallskip
\noindent
\begin{center}
    \large \textbf{The inverse graphical optimization algorithm}
\end{center}
\begin{enumerate}
    \item\label{alg_init}
        Let $\mathcal{X}$ be the set of active boxes in the algorithm,
        $\delta$ be the current accuracy, and $t$ be the current
        objective threshold for the problem. Assume that an initial value of
        $t$ satisfying 
        \[
            t > \min_F f(x)
        \]
        is somehow known from the problem domain, see notes below for details.
        Initialize $\mathcal{X} = \{X\}$ and $\delta = S\delta(X)$
        where $X\in\mathcal{B}$ is any box containing the feasible region 
        and $S$ is the domain scale parameter of the algorithm.
    \item\label{alg_bound}
        Find the tight approximate threshold for the problem as follows.
        \begin{enumerate}
            \item\label{alg_classify}
                Compute $C(X_i, t - \delta)$ for every $X_i\in\mathcal{X}$
                by applying the domain estimation procedure.
            \item If there exists no filled box, i.e.
                $C(X_i, t-\delta)\in\{0,\beta\}$ for all $X_i\in\mathcal{X}$,
                then the tight threshold is found and the algorithm proceeds
                to step~\ref{alg_test}.
            \item Otherwise, prune all the empty boxes from $\mathcal{X}$,
                i.e. replace $\mathcal{X}$ by
                \[
                    \left\{X_i\in\mathcal{X} \mid
                    C(X_i, t-\delta)\in\{\beta, 1\}\right\}.
                \]
            \item Replace $t$ by the new threshold $t-\delta$ and return to
                step~\ref{alg_classify}.
        \end{enumerate}
    \item\label{alg_test} Test for stop condition as follows.
        \begin{enumerate}
            \item\label{alg_test_comp}
                Compute the \emph{actual accuracy estimate},
                i.e. $\delta'$ such that
                \[
                    \delta' \ge t - \min_F f(x).
                \]
                See notes below for details.
            \item\label{alg_test_loop}
                If $\delta' \ge \varepsilon$ then a sufficient accuracy is
                not reached and the algorithm continues to step~\ref{alg_branch}.
            \item\label{alg_result}
                Otherwise, the accuracy of the solution is sufficient.
                So, we stop the algorithm returning $t$
                as the approximate solution value and take any
                point of the last filled box found at step~\ref{alg_bound} as the
                approximate solution point.

                In case there is no filled box found so far we split all the
                boundary active boxes
                similarly to step~\ref{alg_branch} and
                apply domain estimation procedure until a filled box is found.
        \end{enumerate}
\end{enumerate}
\end{minipage}

\noindent
\begin{minipage}{\textwidth}
\begin{enumerate}
    \setcounter{enumi}{3}
    \item\label{alg_branch}
        At this point target accuracy is not reached.
        So, split all the boxes in $\mathcal{X}$ into 2 parts along every coordinate
        axis. This produces $2^n$ boxes from every $X_i\in\mathcal{X}$ where $n$ is
        the domain dimensionality.
    \item\label{alg_iterate}
        Replace $\delta$ by $\frac{\delta}{2}$. Note that in accordance with
        previous step $\delta = S\delta(X_i)$ for any $X_i\in\mathcal{X}$ at
        every step of the algorithm. Return to step~\ref{alg_bound}.
\end{enumerate}
\smallskip
\hrule
\smallskip
\end{minipage}

To complete the algorithm we first discuss the accuracy estimation
at step~\ref{alg_test_comp} and then the choice of the initial upper
bound at step~\ref{alg_init}.

At step~\ref{alg_test_comp} one may use any problem-specific
algorithm. For example, the accuracy estimation may be derived
from the domain estimation procedure internals. One instance of
such a procedure is presented in section~\ref{app_section}.
For the general case we develop
the following algorithm using the domain estimation procedure only.
All the notions and details are the same as in the main algorithm.

\noindent
\begin{minipage}{\textwidth}
\medskip
\hrule
\smallskip
\noindent
\begin{center}
    \large \textbf{The accuracy estimation procedure}
\end{center}
\begin{enumerate}
    \item Let $\delta'$ be the current accuracy estimate.
        Initialize $\delta' = \delta$.
    \item\label{estimate_loop} Compute $C(X_i, t - \delta')$ by applying the
        domain estimation procedure.
    \item If there are non-empty boxes then replace $\delta'$ by
        $\delta' + \delta$ and return to step~\ref{estimate_loop}.
    \item Return $\delta'$ as the accuracy estimate.
\end{enumerate}
\smallskip
\hrule
\smallskip
\end{minipage}

Notice that
we only use the accuracy estimate $\delta'$ at step~\ref{alg_test_loop}
comparing it to $\varepsilon$.
So, there is no need in computing the correct
value of $\delta'$ in case of $\delta'\ge\varepsilon$, and one could skip
accuracy estimation steps by this condition.

Now we describe the choice of
the initial threshold at step~\ref{alg_init} of the inverse graphical
algorithm.
In most practical problems a reasonable value is known immediately, for example
see the design centering problem in section~\ref{app_section}. Beside that
one can use an objective value at any feasible point if it is easy to compute.
Nevertheless, it is possible to use the inverse method even without the means
of direct computing the objective values or the feasible points.
Assuming that only the domain estimation procedure is
available the following algorithm may be used to compute the initial threshold.
Formal details of this procedure are identical to the main algorithm above.

\noindent
\begin{minipage}{\textwidth}
\medskip
\hrule
\smallskip
\noindent
\begin{center}
    \large \textbf{The initial threshold search procedure}
\end{center}
\begin{enumerate}
    \item
        Initialize $t = 0$, $\mathcal{X} = \{X\}$,
        $\delta = S\delta(X)$.
    \item\label{init_estimate}
        Compute $C(X_i, t)$ for every $X_i\in\mathcal{X}$
        by applying the domain estimation procedure.
    \item
        If there exists a filled box in $\mathcal{X}$ then
        stop the algorithm returning $t$.
    \item\label{it_increase}
        Otherwise, replace $t$ by $t + \delta$.
    \item\label{it_loop}
        If all the boxes in $\mathcal{X}$ are empty than
        return to step~\ref{init_estimate}.
    \item\label{it_split}
        Otherwise, split every box $X_i\in\mathcal{X}$ into $2^n$ equal parts,
        replace $\delta$ by $2\delta$ and return to step~\ref{init_estimate}.
\end{enumerate}
\smallskip
\hrule
\smallskip
\end{minipage}

We prove later that these procedures always return correct result.
However, they will not necessarily finish in finite time for any
domain estimation algorithm and they may not be the best ones for a particular
application. We leave this part of study for the future
research as it highly depends on the properties of
a particular problem and domain estimation procedure.

\subsection{Formal study}
Let us proceed to the formal study of the inverse graphical optimization
algorithm. The main aim is to settle the convergence of the inverse graphical
method to the actual problem solution with theorems
\ref{optimum_theorem} and \ref{finish_theorem}.

Let us start with the proof of the algorithm overall correctnes.

\begin{lemma}
    \label{bound_lemma}
    Consider the inverse graphical optimization algorithm for the problem
    of form (\ref{problem}). The current threshold $t$ in the algorithm is
    always an upper bound for the problem, i.e.
    \begin{equation}
        \label{threshold_equation}
        t > \min_F f(x)
    \end{equation}
    at any step of the algorithm.
\end{lemma}
\begin{proof}
    At the beginning of the algorithm (\ref{threshold_equation}) holds
    by definition.
    After that $t$ is only changed at step~\ref{alg_bound} of the algorithm.
    Note that we replace $t$ by $t' = t - \delta$ only if there are filled boxes
    for $t'$, i.e. there exist at least one point $x$ such that $x\in F$ and
    \[
        t' > f(x) \ge \min_F f(x).
    \]
    So, (\ref{threshold_equation}) holds at any step of the algorithm.
\end{proof}

This immediately leads us to the first convergence theorem.

\begin{theorem}
    \label{optimum_theorem}
    Consider the inverse graphical optimization algorithm for the problem
    of form (\ref{problem}).
    Then the following holds for any target accuracy $\varepsilon$ if the
    algorithm actually finishes:
    \begin{enumerate}
        \item The approximate solution value $f_\varepsilon$ returned by
            the algorithm is a global
            $\varepsilon$-op\-ti\-mal solution \cite{Tuy2016}, i.e.
            \[
                \left|f_\varepsilon - \min_F f(x) \right| \le \varepsilon.
            \]
            Note that this implies
            \[
                f_\varepsilon\to \min_F f(x)
            \]
            as $\varepsilon\to 0$.
        \item The sequence of the approximate solution points returned
            by the algorithm is feasible,
            has an accumulation point and every accumulation point $x'$ of this
            sequence is the solution point for the problem (\ref{problem}),
            i.e. $x'\in F$ and
            \[
                f(x') = \min_F f(x).
            \]
    \end{enumerate}
\end{theorem}
\begin{proof}
    First notice that the algorithm can stop only at step~\ref{alg_result}
    and only when
    \[
        t - \min_F f(x) \le \delta' < \varepsilon.
    \]
    Moreover, by lemma~\ref{bound_lemma}
    \[
        t > \min_F f(x)
    \]
    at any step of the algorithm.
    Therefore, for $f_\varepsilon = t$ returned as the solution we have
    \[
        \left|f_\varepsilon - \min_F f(x)\right| < \varepsilon.
    \]
    This completes the first statement of the theorem.
    In addition, notice that any solution point
    $x_\varepsilon$ returned at step~\ref{alg_result} lies inside the filled
    box for the corresponding threshold, so $x_\varepsilon\in F$ and
    \[
        \min_F f(x) \le f(x_\varepsilon) \le f_\varepsilon <
        \min_F f(x) + \varepsilon.
    \]
    Moreover, the feasible region $F$ is a compact set by the problem
    definition. Consequently the infinite sequence of
    \[
        \{x_\varepsilon \mid \varepsilon\to 0\}
    \]
    has a feasible accumulation point $x'$ and for every such point we have
    \[
        f(x') = \min_F f(x),
    \]
    which is essentially the second statement of the theorem.
\end{proof}

To complete the convergence theory it remains to show that 
the inverse graphical algorithm stops in finite time. This part requires
some restrictions on estimation procedures.

First, recall that a domain estimation procedure is allowed by
definition to label any empty or filled box as boundary. So, the trivial example
of such a procedure maps any box to the boundary value $\beta$. Obviously,
such a procedure does not account the problem at all and thus cannot lead to a
solution. This argument motivates a restriction on domain estimation procedures.
We introduce it by the following definition.

\begin{definition}
    \label{approx_def}
    Given the problem of form (\ref{problem}) and a domain estimation
    procedure $C$ we call the procedure $C$ \textbf{problem approximating} if
    for every $x\in\sint(F)$ and $t\in\mathbb{R}$ whenever $f(x) < t$
    there exists $\delta > 0$ such that
    \[
        C(B, t) = 1
    \]
    for any $B\in\mathcal{B}$ whenever $x\in B$ and $\delta(B) < \delta$.
\end{definition}

Intuitively, a problem approximating domain estimation procedure
recognize the filled boxes precisely
when the box sizes becomes sufficiently small.
Note that even for the absolutely precise domain estimation procedure every box
containing a point $x\in\partial F$ is boundary. Consequently,
when the point tends to the feasible region boundary the corresponding box
size $\delta$ needed for precise labeling tends to 0. This explains
the definition based on interior feasible points.

Now we need some restriction on the actual accuracy estimate at
step~\ref{alg_test_comp} of the algorithm. Notice that one can use too big
actual accuracy estimate with $\delta' > \varepsilon$ holding always.
Such accuracy estimate results
in the algorithm running infinitely, so we introduce the following definition
to avoid this.

\begin{definition}
    \label{convergent_definition}
    Consider the inverse graphical optimization algorithm for the problem
    of form (\ref{problem}).
    We say that the actual accuracy estimation procedure is
    \textbf{convergent} if for any $\varepsilon > 0$ there exist such
    $\delta(\varepsilon)$ and $\theta(\varepsilon)$ that
    the obtained accuracy estimate $\delta' > 0$ satisfies
    \[
        \delta' < \varepsilon
    \]
    whenever the algorithm does not finish before and reaches
    $\delta < \delta(\varepsilon)$ along with
    \[
        t < \min_F f(x) + \theta(\varepsilon).
    \]
\end{definition}

This definition means that the computed estimate of the actual accuracy
tends to 0 if the current accuracy in the algorithm tends to 0 and
the current thresholds tends to the actual optimal solution.

Having all the definitions we complete the inverse graphical algorithm
convergency theory with the following theorem.

\begin{theorem}
    \label{finish_theorem}
    Consider the inverse graphical optimization algorithm for the problem
    of form (\ref{problem}).
    Suppose the domain estimation procedure
    is problem approximating and the accuracy estimation procedure is
    convergent. Then the algorithm finishes after finite number of steps
    for any target accuracy $\varepsilon$.
\end{theorem}

The proof of theorem~\ref{finish_theorem} is given in appendix
section~\ref{proofs_section}.
To finalize the theoretical results we provide the basic correctness statements
for the supporting routines of the inverse graphical algorithm.

\begin{theorem}
    \label{accuracy_theorem}
    Consider the inverse graphical optimization algorithm for the problem
    of form (\ref{problem}) and the accuracy estimation procedure described
    in section~\ref{algorithm_subsection}. If the accuracy estimation
    procedure finishes with value $\delta'$ then $\delta'$
    is the actual accuracy estimate, i.e.
    \[
        \delta' \ge t - \min_F f(x).
    \]
\end{theorem}

\begin{theorem}
    \label{initial_theorem}
    Consider the inverse graphical optimization algorithm for the problem
    of form (\ref{problem}) and the initial threshold search procedure
    described in section~\ref{algorithm_subsection}. If this procedure
    finishes with value $t$ then $t$ is an upper bound for the problem,
    i.e.
    \[
        t > \min_F f(x).
    \]
    Moreover, if the domain estimation procedure is problem approximating,
    then the initial threshold search procedure finishes in finite number
    of steps.
\end{theorem}

The proofs can be found in appendix section~\ref{proofs_section}.
Combining these results with theorem~\ref{optimum_theorem} we deduce, that
even when utilizing solely the domain estimation procedure the inverse graphical
algorithm always produces $\varepsilon$-optimal solution whenever it finishes
in finite time. For the algorithm to
stop inevitably it only remains to show that the accuracy estimation procedure
is convergent for a particular problem. As was stated above, we leave this to
the research of applications.

\section{Solving the design centering problem}
\label{app_section}
In this section we study an application of the general framework developed above
to a special design centering problem.
First, the problem statement and a brief review of
the previous work is provided. Then we discuss the implementation details
of the state-of-the-art methods along with the inverse approach.
At the end the practical comparison of the algorithms is presented.

Let us start with the problem description. We consider the design centering
problem arising in the field of diamond manufacturing
\cite{Nguyen1992}, \cite{Horst1996}. Precisely, in the diamond industry
an important problem is to cut the largest diamond of a prescribed shape from
a given rough stone. Following \cite{Nguyen1992} we limit the allowed
transformations to translation and scaling,
so the formal problem stament is as follows.
Let $Q\subset\mathbb{R}^n$ be a nonempty compact set. In addition,
let $K\subset\mathbb{R}^n$ be a nonempty compact \textbf{star-shaped}
set. The latter means that the whole boundary $\partial K$ is visible
from a single interior point. Without loss of generality we assume this
point to be origin $0\in\mathbb{R}^n$. Under these assumptions the
problem is to find \textbf{center} $x\in\mathbb{R}$ and \textbf{radius}
$r\ge 0$ solving the mathematical program
\begin{equation}
    \label{orig_dc_problem}
    \begin{gathered}
        r \to \max \\
        \mbox{s.t.}\; x + r K\subset Q,
    \end{gathered}
\end{equation}
see Fig. \ref{dc_figure}.

\begin{figure}[h]
    \begin{center}
        \includegraphics[width=0.6\textwidth]{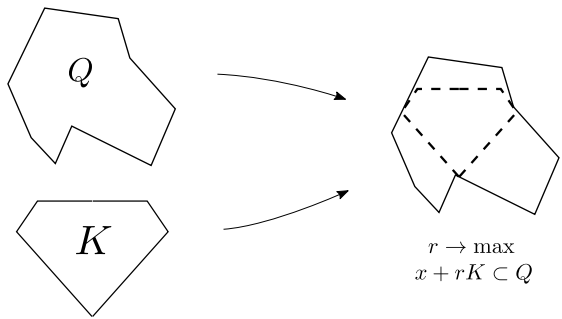}
    \end{center}
    \caption{Design centering problem}
    \label{dc_figure}
\end{figure}

For convenience we use the following equivalent statement when needed:
\begin{equation}
    \label{dc_problem}
    \begin{gathered}
        r(x) = \max\{r\ge 0 \mid x+rK\subset Q\} \to \max\\
        \mbox{s.t}\; x\in Q.
    \end{gathered}
\end{equation}
In the rest of the article we call $Q$ a \textbf{contour} and
$K$ a \textbf{pattern} in the sense of searching for a set similar to $K$
inside $Q$. In addition, we call function $r(x)$ defined above a
\textbf{radius value}.
Note that we depend on a weaker star-shapeness property of $K$ as opposed
to the convexity in \cite{Nguyen1992} and \cite{Horst1996}. The practical
importance of this is obvious when considering the widely
used heart diamond cut which is essentially star-shaped and non-convex,
see Fig.~\ref{patterns_figure}~(b).
In addition,
we assume when needed through this section that $Q$ and $K$ are polytopes.
In practice, it is always the case because the real stone models
are represented and visualized by the computers as the sets of simplex
faces constituting the polytope boundary. Accordingly, by polytope
we denote the compact connected region bounded by a finite set of
$n-1$ simplices.
Note that the diamond industry deals with 3-dimensional bodies only,
so we study this case primarily. However, most of the research remains valid
for higher dimensions as well.

Let us continue with a review of the previous work related to the problem.
In \cite{Nguyen1992} the problem is reduced to
a difference of convex programming problem. It is assumed there that
$K$ is convex and $Q$ is an intersection of a convex set with a finite
number of complementary convex sets. In that case the problem can be solved
by the
combination of the outer approximation and successive partition methods
\cite{Thoai1988}. Note that it is essential for these algorithms that
$Q$ and $K$ satisfy the convexity requirements. As a consequence, such
approaches are not directly applicable to a problem we consider
in this article. Nevertheless, a useful algorithm computing the radius value
for problem (\ref{dc_problem}) was developed in \cite{Nguyen1992}.

The related study for a general design centering problem is provided in
\cite{Thach1988}. It is proved there that the problem is globally Lipschitzian
whenever $K$ is convex. Therefore, in addition to the scheme of
\cite{Thoai1988} any method of Lipschitzian global optimization can be used.
Moreover, in practice contour
$Q$ is usually given by a set of simplex faces, i.~e. without a direct
difference of convex decomposition. In such a case using a general Lipschitzian
method may be the only possible approach \cite{Thach1988}.
In particulal,
the state of the art branch-and-bound framework can be easily applied to this
kind of problem \cite{Scholz2012}. So, the design centering problem under
consideration is a good candidate to test the straight and the inverse
solution methods simultaneously.
In \cite{Horst1995} a performance research is provided for a variety of
branch-and-bound based methods, we use it to choose a good technique to
compare to. In addition, we compare the inverse method to the widely used
and recent Lipschitzian optimization techniques \cite{Jones1993},
\cite{Gablonsky2001}, \cite{Malherbe2017}. Note that we do not involve into
comparison advanced bisection methods \cite{Abaffy2018},\cite{Shary2001}.
The reason is high computational complexity of solving equations
$r(x) = c$ for objective functional in (\ref{dc_problem}) resulting from
nonsmoothness of $r(x)$. However, it is still worth looking at various
equation solving techniques, we leave this for future research.

The supporting algorithms described below for the inverse
approach are based on the
Min\-kow\-ski sum concept and voxelization techniques. Let us note the related
works. An application of the Minkowski sums to the geometric placement
problems related to the one under consideration can be found in
\cite{Pustylnik2007}. However, it deals with the patterns of fixed size only,
i.e. variable radius in problem (\ref{orig_dc_problem}) is not considered.
As for the Minkowski sums computation, the good algorithms are developed in
\cite{Sacks2013}, \cite{Varadhan2006}, \cite{Li2011}. In our appproach the
complete Minkowski sum computation can be omitted, so we use a custom
technique inspired by the spatial tree labelling \cite{SamHaroud1995} and based
on the general convex decomposition \cite{Varadhan2006}
along with the volumetric geometry representation \cite{Kaufman1994}.

In addition to the above, the design centering problem arising from diamond
industry can be treated as the geometric placement problem. Examples of the
related research can be found in
\cite{Deits2014,Chebrolu2008,Denny2003,Winterfeld2008}.
However, all the techniques discovered by the author so far
either are limited to 2-dimensional case only or depend on the convexity of the
pattern. In this article we extend the class of problems to solve by allowing
non-convex star-shaped patterns in $\mathbb{R}^3$ and the core results
can be applied to the problems of an arbitrary number of dimensions.

\subsection{Implementation of conventional methods}
\label{classic_impl_section}
Let us describe the application of general optimization schemes
to problem (\ref{dc_problem}). First of all we need to compute
the radius value at any given point, which is done by means of
\cite{Nguyen1992}. In addition, we have to
ensure that the points selected actually lie inside
the contour. This is done by the well-known ray-casting algorithm
assuming the contour to be a polytope.

Now consider application of the branch-and-bound technique.
Note that we study a maximization
problem, so the upper and lower bounds may be the opposite to the ones
in references.
The general branch-and-bound scheme \cite{Horst1996} depends on three
operations: bounding, selection and refining. Various selection and
refining procedures may be used in the inverse method as well as in
branch-and-bound ones. For the sake of simplicity we
have used the partition into boxes and selection of all the active boxes
in the inverse framework above, so we do for the branch-and-bound method
implementation. The main question remains is chosing the bounding operations.

The common way to determine a lower
bound is computing the objective value at some point inside the region
\cite{Scholz2012}, it has been discussed already.
As for the upper bounds, notice that
the radius value depending on center $x$ in problem (\ref{dc_problem})
is a Lipschitz function. For the case of convex pattern
it was proved in \cite{Thach1988}. We extend this result to the star-shaped
patterns in order to compare the conventional and new approaches.

\begin{theorem}
    \label{lipschitz_theorem}
    Consider the problem of form (\ref{dc_problem}). Suppose the pattern
    is a star-shaped polytope and no pattern face is coplanar with pattern
    origin $0$. Then radius value $r(x)$ is a Lipschitz function with constant
    \[
        L = \frac{1}{\Delta},
    \]
    where $\Delta$ is the minimum distance from origin $0$ to hyperplanes
    containing pattern faces.
\end{theorem}

The proof can be found in appendix section~\ref{proofs_section}.
As a consequence from theorem~\ref{lipschitz_theorem},
we can use a Lipschitzian bounding operation \cite{Scholz2012}
in the branch-and-bound method. Moreover, the lemma statement provides us the
precise Lipchitz constant.

\subsection{Implementation of the inverse graphical optimization scheme}
\label{inverse_impl_section}
Now we move to the implementation of the inverse graphical method.
By the theory before, we need to specify a domain estimation procedure
satisfying the problem approximating property along with the convergent accuracy
estimation algorithm.

Let us start with developing the domain estimation procedure.
The key observation
allowing the application to the design centering problems is a relation between
the translated sets inclusion and a morphology operation called \emph{erosion}
\cite{Pustylnik2007,Serra1982}. The latter can be expressed in terms of
\emph{Minkowski sums}. We use the following definition.
\begin{definition}
    Let $A$ and $B$ be sets in $\mathbb{R}^n$. The \textbf{Minkowski sum}
    of $A$ and $B$ is defined as
    \[
        A \oplus B = \{a + b \mid a\in A,\,b\in B\}.
    \]
    The \textbf{erosion} of $A$ by $B$ is defined as
    \[
        A \ominus B = A \setminus \left(\partial A \oplus B\right).
    \]
\end{definition}

The Mikowski sums have been extensively used for
solving motion planning problems \cite{Perez1983,Aronov1997,Lien2008}.
Following this study, in scope of problem (\ref{dc_problem})
\[
    x + r K \subset Q
\]
if and only if
\begin{equation}
    \label{erosion_relation}
    x \in Q \ominus r(-K) = Q \setminus \left(\partial Q \oplus r(-K)\right)
\end{equation}
where
\[
    -K = \{-k \mid k\in K\}
\]
is the \textbf{reflection} of $K$. See Fig. \ref{erosion_figure} for details.

\begin{figure}[h]
    \begin{center}
        \includegraphics[width=0.8\textwidth]{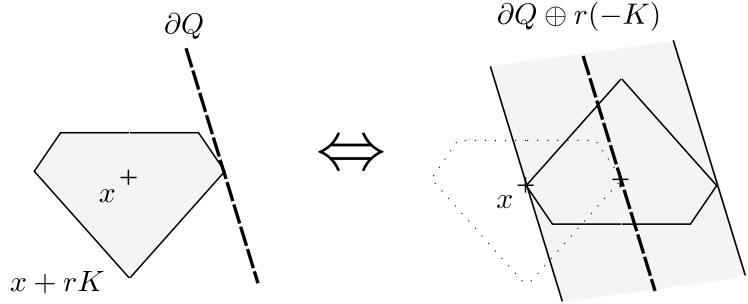}
    \end{center}
    \caption{Relation between design centering and morphology}
    \label{erosion_figure}
\end{figure}

Now let us define domain estimation procedure $C(B,t)$ for
box $B\in\mathcal{B}$ and threshold $t\in\mathbb{R}$ as follows:

\begin{equation}
    \label{dc_classify}
    \begin{aligned}
        C(B,t) &= 0 \mbox{ if $B\cap Q = \emptyset$ or
            $B\subset \left(\partial Q \oplus t(-K)\right)$},\\
        C(B,t) &= 1 \mbox{ if $B \subset Q$ and
            $B\cap \left(\partial Q\oplus t(-K)\right) = \emptyset$},\\
        C(B,t) &= \beta \mbox{ in any other or unknown case}.
    \end{aligned}
\end{equation}
Notice that for star-shaped patterns $r_1 (-K) \subset r_2 (-K)$
whenever $r_2 > r_1 > 0$. Combining this with (\ref{erosion_relation})
it can be easily seen that
$C(B,t)$ is a domain estimation procedure for problem (\ref{dc_problem}).

Intuitively our domain estimation procedure erodes the contour by the
reflection of the pattern. The boxes lying too close to the
contour boundary and not containing a good solution to the problem are
marked empty.
Eventually we shrink the search area to a small region containing
the optimal solution. Formal study of this procedure is presented later in
this section.

Let us continue with practical implementation of (\ref{dc_classify}).
In this section the main references are provided, see \cite{SiteXDScribe}
for particular implementation details.
Notice that the domain estimaiton procedure in the inverse graphical
optimization
algorithm is already defined on equally-sized boxes, so in fact we need
volumetric images \cite{Kaufman1994} of Minkowski sums and the contour.
For the Minkowski sums we apply the general scheme based on convex
decompositions \cite{Varadhan2006}.
First, we compute the convex decomposition of polyhedra and reduce the
overall Minkowski sum to the union of pairwise convex ones.
For the pattern decomposition we implement the surface flood-fill algorithm
inspired by works \cite{Chazelle1997,Ehmann2001}; the contour boundary is
used without
any preprocessing as the set of simplex faces. Then we compute pairwise convex
Minkowski sums with the help of CGAL \cite{SiteCGAL}. After that
we need to acqire the volumetric images of individual convex Minkowski sum.
For this we involve a straightforward algorithm considering convex polytope as
an intersection of half-spaces. As for the contour image, we first
label the boundary boxes by a surface voxelization algorithms \cite{Schwarz2010}
and then separate inner and outer ones by testing a point inside every box.
The latter is done with the help of simultaneous ray-casting procedure,
which is essentially equivalent to a distance field computation
along single direction \cite{Kobbelt2001}.

Although some different techniques of Minkowski sums voxelization are
known \cite{Li2011,Barki2009}, the
method described above is rather easy to implement and fits perfectly
into the inverse graphical method. The performance is good enough as well,
see comparison at the end of the article.

The result domain estimation algorithm for threshold $t$ is as follows,
refer to Fig. \ref{domain_estimation_figure} and \cite{SiteXDScribe} for
details.

\begin{figure}[h]
    \begin{center}
        \includegraphics[height=0.6\textheight]{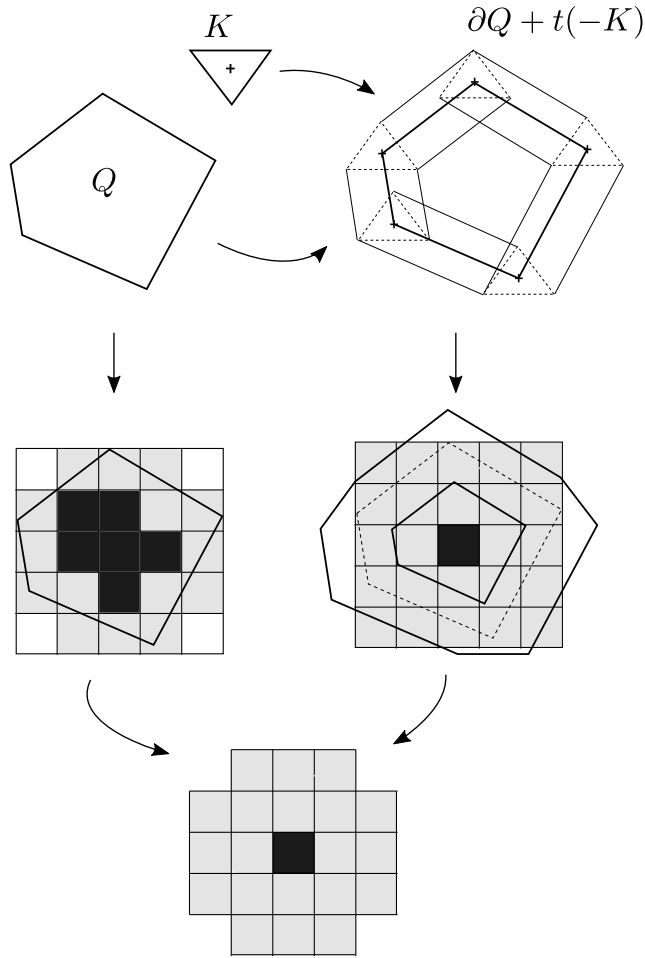}
    \end{center}
    \caption{Domain estimation procedure for design centering}
    \label{domain_estimation_figure}
\end{figure}

\noindent
\begin{minipage}{\textwidth}
\medskip
\hrule
\smallskip
\noindent
\begin{center}
    \large \textbf{The domain estimation procedure for design centering}
\end{center}
\begin{enumerate}
\item Decompose the contour boundary $\partial Q$ and the reflection of
    pattern $-K$ into convex parts.
\item Scale the reflected pattern convex parts by $t$.
\item Compute all the pairwise Minkowski sums of convex parts by the convex-hull
    based algorithm. We call these sums
    \emph{convex elements} of $\partial Q \oplus t(-K)$.
\item\label{vox_elements}
    Compute voxelizations of the convex elements on active boxes.
    This means that for
    every box under consideration in the inverse graphical algorithm
    we decide whenever it is inner, outer or boundary relative to the
    convex element.
\item\label{vox_outside}
    Analogiously, we compute the voxelization of the contour itself.
\item\label{vox_unite}
    Unite all the voxelizations. This is done by simple boolean operations on
    the results of previous steps for every box. By 
    (\ref{dc_classify}) we put $C(B,t) = 0$ if $B$ is outside the contour
    at step~\ref{vox_outside} or $B$ is inside any element at
    step~\ref{vox_elements}; $C(B,t) = 1$ if $B$ is inside the contour at
    step~\ref{vox_outside} and $B$ is outside all the elements at
    step~\ref{vox_elements}; $C(B,t) = \beta$ in all the remaining cases.
\end{enumerate}
\smallskip
\hrule
\smallskip
\end{minipage}

The theoretical result is given by the following theorem.

\begin{theorem}
    \label{domain_estimation_theorem}
    The domain estimation procedure defined by the algorithm above is
    problem-approximating.
\end{theorem}

See appendix section~\ref{proofs_section} for the proof.

Let us continue with constructing a convergent accuracy estimation procedure
for the design centering problem under consideration.
Although the general accuracy entimation procedure described in
section~\ref{algorithm_subsection} could be used, it is a bit complicated to
prove its convergency. Nevertheless, it still gives a correct result by
theorem~\ref{accuracy_theorem} and even provided a better
performance in practice, refer to section~\ref{comparison_section} for
details.
In this section we propose a simple method based on the Lipschitz property
of the radius value to make a theoretically complete implementation.
Precisely, the following theorem holds.

\begin{theorem}
    \label{lipschitz_accuracy}
    Consider a problem of form (\ref{dc_problem}) and the inverse graphical
    optimization algorithm. Let $L$ be the Lipschitz constant of
    the radius value by theorem~\ref{lipschitz_theorem}, $S$ be
    the domain scale and $\delta$ be
    the current accuracy in the algorithm. Then
    \begin{equation}
        \label{dc_acc_estimate}
        \delta' = (1 + \frac{L}{S})\delta
    \end{equation}
    is the convergent actual accuracy estimate.
\end{theorem}

For the proof refer to appendix section~\ref{proofs_section}.
Note that the value of $L$ needed for (\ref{dc_acc_estimate})
is easy to compute due to theorem~\ref{lipschitz_theorem}. Now all the
supporting algorithms required for the practical implementation of
the inverse graphical optimization method for problem (\ref{dc_problem})
are specified.

\subsection{Comparison results}
\label{comparison_section}
In this section we discuss the practical performance of the inverse and
straight approaches to the diamond cutting problem described above.
The algorithms have been implemented
in C++ programming language and tested on various data using a machine
with Core i7-8750H CPU and 16Gb RAM. For the
implementation details refer to the source code published by the author
\cite{SiteXDScribe}. We emphasize that only
single-threaded CPU-based implementations have been tested as we aim on
comparison of the general approaches.
It is worth noting that one can apply more optimized algorithms
mentioned above and in the references, especially parallelized
and GPU-based ones. We leave these ideas for the future research.

The testing data used in comparison consists of publicly available diamond
cut models \cite{SiteLapidary}
along with generated rough stone models of different resolution and complexity.
Some examples are shown in Fig. \ref{patterns_figure} and
Fig. \ref{contours_figure}.
The precise data can be found at \cite{SiteXDScribe}.
Note that we involve non-convex patterns as opposed to the traditional problem
research. Besides, we have included some defects and irregularities in the
rough stone models as they are common in the real-life industry

\begin{figure}[h]
    \begin{center}
        \includegraphics[width=0.7\textwidth]{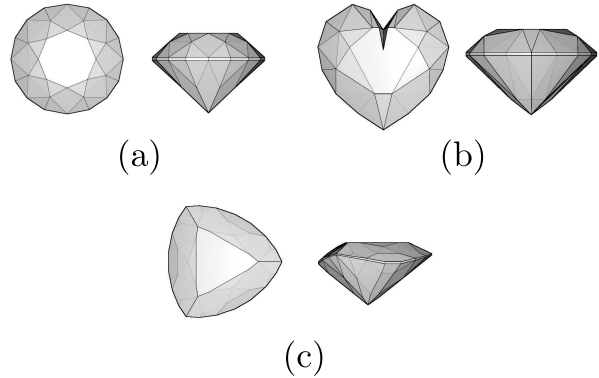}
    \end{center}
    \caption{Diamond cut models:
        (a)~standard\_126, (b)~rosehrttrue\_104, (c)~1stwave\_172}
    \label{patterns_figure}
\end{figure}

\begin{figure}[h]
    \begin{center}
        \includegraphics[width=0.9\textwidth]{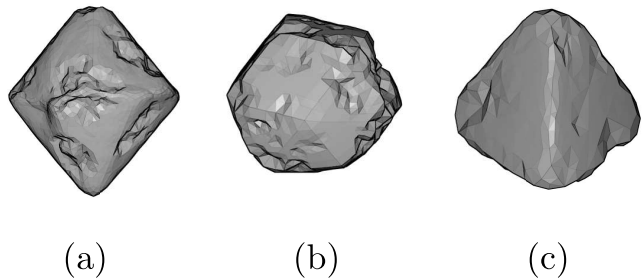}
    \end{center}
    \caption{Rough stone models:
        (a)~octa\_rough\_6432, (b)~rhombic\_rough\_2304, (c)~tetra\_rough\_912}
    \label{contours_figure}
\end{figure}

Let us describe briefly the particular algorithms of global numerical
optimization we compare. First, we have choosen a branch-and-bound variant due
to Gourdin, Hansen and Jaumard, as it is stated to be the fastest
in \cite{Horst1995}. We denote it by "GHJ" later in this section.
Next, we consider modern Lipschitzian optimization schemes, not related
to the branch-and-bound directly. The ones available in the open libraries are
the DIRECT algorithm \cite{Jones1993,Gablonsky2001} as
provided by the nlopt library \cite{SiteNlopt}
and the LIPO scheme \cite{Malherbe2017} implemented by the
dlib C++ library \cite{SiteDlib}. Both algorithms are direct-search ones
by nature. Recent research \cite{Malherbe2017} indicates that
their performance is competitive, however we were unable to reach meaningful
results with LIPO. The reason seems to be in random sampling of the domain
which takes too much time finding feasible points. Additionally, it is quite
hard to implement a correct stopping criterion to get a proven optimal result.
So we do not present this method
in our research and refer the readers to \cite{Malherbe2017} for a relative
comparison with DIRECT and to \cite{SiteXDScribe} for the implementation and
testing.
As for DIRECT algorithm, the main problem is the absense of the
proven precision estimations in nlopt library.
So we have adapted the Lipschitzian bounds in the stopping criterion.
For the comparisons we have choosen
the best variants provided by the library, namely scaled DIRECT and
locally-biased non-scaled DIRECT with randomization. We refer to them as
"DIRECT" and "local DIRECT" respectively.

Concerning the inverse graphical method we include two variants in comparison
differing only by the accuracy estimation procedure. As was described above,
we implement the general scheme from section~\ref{gen_section} as well as
a Lipschitzian accuracy estimation procedure provided by
theorem~\ref{lipschitz_accuracy}. We denote these variants as
"general inverse" and "Lipschitzian inverse" respectively.

Let us proceed to the actual results. First, we compare the running
times while searching for a proven global
optimal solution with a specified precision. The results are presented
in Table~\ref{runtime_proven} with "$>$ 5 hours" meaning that the
algorithm failed to reach the target accuracy within five hours.
As we can see, the inverse methods are superior, especially on hard
problems and high precision requested. This remains true even when
using the same Lipschitzian accuracy estimation as within the straight 
algorithms.

\begin{table}[h]
    \center
    \resizebox{\textwidth}{!}{
    \begin{tabular}[c]{>{\bfseries}LCCCCC}
        \hline
        \textbf{Problem} & \textbf{GHJ} & \textbf{DIRECT} &
        \textbf{local DIRECT} &
        \textbf{general inverse} & \textbf{Lipschitzian inverse} \\
        \hline \\
        standard\_126 rhombic\_rough\_576 1e-2 &
        99 & 292 & 478 & 12 & 15 \\ \\
        standard\_126 rhombic\_rough\_576 1e-3 &
        336 & 1665 & 2375 & 41 & 49 \\ \\
        standard\_126 rhombic\_rough\_576 1e-4 &
        679 & 4663 & 6638 & 62 & 69 \\ \\
        bstilltrue\_194 \mbox{rhombic\_rough\_2304} 1e-2 &
        $>$ 5 hours & $>$ 5 hours & $>$ 5 hours & 193 & 638 \\ \\
        rosehrttrue\_104 tetra\_rough\_912 1e-2 &
        5329 & 10347 & $>$ 5 hours & 69 & 78 \\ \\
        rosehrttrue\_104 tetra\_rough\_912 1e-3 &
        11545 & $>$ 5 hours & $>$ 5 hours & 84 & 89 \\ \\
        1stwave\_172 tetra\_flat\_3648 \newline 1e-3 &
        1568 & $>$ 5 hours & $>$ 5 hours & 160 & 165 \\ \\
        novice7\_86 octa\_rough\_6432 1e-2 &
        1688 & 4526 & 5680 & 359 & 672 \\ \\
        novice7\_86 octa\_rough\_6432 1e-3 &
        9434 & $>$ 5 hours & $>$ 5 hours & 1235 & 1268 \\ \\
        \hline
    \end{tabular}}
    \caption{Runtimes acquiring proven solution, seconds}
    \label{runtime_proven}
\end{table}

Next, we compare the time needed to accomplish the specified solution,
see Table~\ref{runtime_reach}.
Now we can see that DIRECT methods are reaching the
good points earlier during computation.
However, it is worth noting that in this case we need to know beforehand
the optimal value we are actually looking for to stop the algorithm,
otherwise we have no measure of the result quality.

\begin{table}[h]
    \center
    \resizebox{\textwidth}{!}{
    \begin{tabular}[c]{>{\bfseries}LCCCCC}
        \hline
        \textbf{Problem} & \textbf{GHJ} & \textbf{DIRECT} &
        \textbf{local DIRECT} &
        \textbf{general inverse} & \textbf{Lipschitzian inverse} \\
        \hline \\
        standard\_126 rhombic\_rough\_576 v1.429 &
        97.18 & 7.70 & 3.44 & 37.40 & 37.54 \\ \\
        bstilltrue\_194 \mbox{rhombic\_rough\_2304} v0.312 &
        $>$ 5 hours & 10.69 & 8.61 & 369.30 & 369.45 \\ \\
        rosehrttrue\_104 tetra\_rough\_912 v0.288 &
        10147.29 & 14.43 & 24.57 & 94.52 & 95.04  \\ \\
        1stwave\_172 tetra\_flat\_3648 \newline v5.089 &
        1186.32 & 117.70 & 34.67 & 150.76 & 151.03 \\ \\
        novice7\_86 octa\_rough\_6432 v1.279 &
        1587.34 & 87.43 & 84.09 & 1260.66 & 1247.45 \\ \\
        \hline
    \end{tabular}}
    \caption{Runtimes reaching objective value, seconds}
    \label{runtime_reach}
\end{table}

Combining all the results together we can conclude that the inverse
graphical method is very promising for finding precise global optimal solutions
to the optimization problems, and indeed provide ones within
a practically tractable time. However, when one needs any good solution
quickly regardless of its global quality, it is sensible to use something
like DIRECT method or even local search.

In addition, note that one need to know the Lipschitz constant of the objective
to get the proven global optimum by the straight methods. In contrast,
the inverse scheme with general accuracy estimation does not depend on it
to acquire the guaranteed result. Though it has not been proved
theoretically that the algorithm would always finish in that case, in practice
it does and even shows the best performance.

\section{Conclusions and future work}
\label{con_section}
A novel inverse approarch to deterministic global optimization has been
developed within this article. It is based on domain estimation instead of
objective value computation. We have proved theoretically its convergency and
studied requirements on the procedures involved. The algorithm has been
implemented in practice demonstrating clear advantage of the new method while
searching for a proven global optimum for the design centering problem.
Additionally, it have been shown that the
inverse method is able to reach a guaranteed optimal solution even without
knowing the Lipschitz constant, and the current theoretical results only
depend on the continuity of the objective.

Concerning the application to the design centering, the problem has been
extended to the case of non-convex star-shaped patterns and the complete
implementation of
various optimization methods has been provided.

Among drawbacks and future directions we should note that the inverse scheme 
depends heavily on estimation procedures and particular algorithms.
These procedures differ in
essence from the widely used black-box objective computations,
and they could be hard or
even impossible to implement for particular problems.
So, it would be of high interest to research more applications of the inverse
optimization approach, not necessarily the graphical one.
Besides, it still remains to study sufficient
conditions for the inverse graphical method to converge with the general
accuracy estimation procedures. Nevertheless, the inverse technique is worth
considering for the practical search of global solutions to optimization problems.

\appendix
\section{Proofs}
\label{proofs_section}
\subsection{Supporting theory}
First, let us prove some correctness lemmas for the inverse graphical
optimization algorithm.

\begin{lemma}
    \label{subbox_lemma}
    Consider the problem of form (\ref{problem}), any domain estimation
    procedure $C$ and any threshold $t\in\mathbb{R}$ satisfying
    \[
        t > \min_F f(x).
    \]
    Suppose box $B$ is empty under $C$ for threshold $t$;
    then for every threshold $t' \le t$ every box $B'\subset B$
    is empty as well.
\end{lemma}
\begin{proof}
    By the definition of the domain estimation procedure
    for every point $x\in B$ either $x\not\in F$ or
    $f(x) > t$. Consequently, for every point $x'\in B'\subset B$ either
    $x'\not\in F$ or $f(x') > t \ge t'$. So, $B'$ is empty by definition.
\end{proof}

\begin{corollary}
    \label{feasible_corollary}
    Consider the inverse graphical optimization algorithm for the problem
    of form (\ref{problem}). Then the set of active boxes always covers the
    feasible set for the current threshold, i.e.
    \begin{equation}
        \label{feasible_equation}
        F(t) = \{x\in F \mid f(x) < t\} \subset \bigcup_\mathcal{X} X_i
    \end{equation}
    at any step of the algorithm.
\end{corollary}
\begin{proof}
    At the beginning (\ref{feasible_equation}) holds obviously.
    After that $t$ is nonincreasing and we prune the empty boxes only, i.e.
    the ones containing no feasible points satisfying $f(x) < t$.
    It then follows from lemma~\ref{subbox_lemma} that (\ref{feasible_equation})
    holds at any step of the algorithm.
\end{proof}

Now we continue with the proof that for every threshold satisfying
$t > \min_F f(x)$ the boxes inside the
feasible set eventually become filled by the inverse graphical algorithm.
This property is crucial for the convergence of the algorithm iterations
at steps~\ref{alg_bound} and \ref{alg_result} proved subsequently.

\begin{lemma}
    \label{filled_lemma}
    Consider the problem of form (\ref{problem}) and
    domain estimation procedure $C$.
    Suppose $C$ is problem approximating;
    then for every threshold $t$ satisfying
    \[
        t > \min_F f(x)
    \]
    there exists $\delta > 0$ such that every cover
    of
    \[
        F(t) = \{x\in F \mid f(x) < t\} \subset F
    \]
    by boxes of size at most $\delta$ contains
    a filled box.
\end{lemma}
\begin{proof}
    Consider the objective function $f$ of the problem. By definition,
    $f$ is continuous and $F$ is compact with non-empty interior,
    so the problem solution
    \[
        f_{min} = \min_F f(x)
    \]
    is reached at some point.
    Assuming $f_{min} < t$ we additionally imply that
    $f(x) < t$ holds at least on some ball inside $F$ centered at
    $f_{min}$.
    Therefore, the set
    \[
        I(t) = \{x\in \sint(F) \mid f(x) < t\} \subset F(t)
    \]
    is nonempty. Consider any $x\in I(t)$ and $t$ as threshold.
    By definition~\ref{approx_def} there exists
    $\delta > 0$ such that every box $B$ of size $\delta(B) < \delta$
    containing $x$ is filled. Every cover of $F(t)$ by boxes of size
    at most $\delta$ contains such a box, so the proof is completed.
\end{proof}

\begin{corollary}
    \label{decreasing_corollary}
    Consider the inverse graphical optimization algorithm for the problem
    of form (\ref{problem}). If the algorithm does not stop then
    \[
        \delta\to 0.
    \]
    In addition, suppose the domain estimation procedure is
    problem approximating. 
    Then for every $\theta > 0$
    after finite number of steps in the algorithm
    \begin{equation}
        \label{convergence_equation}
        t < \min_F f(x) + \theta
    \end{equation}
    if the algorithm does not stop before.
\end{corollary}
\begin{proof}
    First notice that step~\ref{alg_bound} of the algorithm finishes
    after finite number of iterations. Indeed, $t$ is decreased by
    a constant value $\delta$ and by lemma~\ref{bound_lemma}
    it is strictly bounded from below by the problem solution. Consequently,
    if the algorithm does not stop at step~\ref{alg_test}
    then it proceeds infinitely
    through step~\ref{alg_iterate} and
    \[
        \delta\to 0.
    \]

    Now consider any $\theta > 0$ and put
    \[
        t' = \min_F f(x) + \frac{\theta}{2}.
    \]
    Suppose $t > t'$, otherwise the corollary statement already holds.
    By corollary~\ref{feasible_corollary} the set of active boxes covers
    \[
        F(t) = \{x\in F \mid f(x) < t\}
    \]
    at every step of the algorithm. Obviously, $F(t') \subset F(t)$,
    so the same boxes cover $F(t')$. We proved above that box sizes
    \[
        \delta(X_i) = \frac{1}{S}\delta
    \]
    eventually become as small as needed if the algorithm does not stop.
    It then follows from
    lemma~\ref{filled_lemma} that there exists a filled box at some
    iteration of the algorithm for threshold $t'$. Consequently,
    the iteration of step~\ref{alg_bound} continues at least until
    $t - \delta < t'$ equivalent to $t < t' + \delta$. This holds for any
    $\delta$ in the algorithm and $\delta\to 0$, so eventually
    \[
        \delta < \frac{\theta}{2}
    \]
    is satisfied. Combining the inequalities we obtain
    \[
        t < t' + \delta < \min_F f(x) + \frac{\theta}{2} + 
        \frac{\theta}{2} = \min_F f(x) + \theta.
    \]
\end{proof}

\begin{corollary}
    \label{result_corollary}
    Consider the inverse graphical optimization algorithm for the problem
    of form (\ref{problem}). Suppose the domain estimation procedure is
    problem approximating. Then step~\ref{alg_result} of the algorithm finishes
    finding a filled box after finite number of attempts.
\end{corollary}
\begin{proof}
    After every attempt at step~\ref{alg_result} we split all the boundary boxes
    reducing their sizes. By corollary~\ref{feasible_corollary} these boxes
    cover
    \[
        F(t) =  \{x\in F \mid f(x) < t\}.
    \]
    Eventually box sizes become sufficiently small, so 
    it then follows from lemma~\ref{filled_lemma} that there exists a filled box
    and the iterating stops.
\end{proof}

\subsection{Proof of theorem~\ref{finish_theorem}}
\begin{proof}
    Consider target accuracy $\varepsilon > 0$ and corresponding
    $\delta(\varepsilon)$, $\theta(\varepsilon)$ from definition
    \ref{convergent_definition}.
    By corollary~\ref{decreasing_corollary} after finite number of steps within
    the inverse graphical optimization algorithm we obtain
    \[
        t < \min_F f(x) + \theta(\varepsilon)
    \]
    and $\delta < \delta(\varepsilon)$ if the algorithm does not stop before.
    Hence, by definition \ref{convergent_definition}
    the convergent accuracy estimation procedure at step~\ref{alg_test_comp}
    of the algorithm produces
    \[
        \delta' < \varepsilon
    \]
    after finite number of iterations. The algorithm then
    proceeds to step~\ref{alg_result} which finishes in finite number of
    steps by corollary~\ref{result_corollary}. So, the whole algorithm
    stops.
\end{proof}

\subsection{Proof of theorem~\ref{accuracy_theorem}}
\begin{proof}
    Indeed, the accuracy estimation procedure stops only when the result of
    domain estimation with threshold $t - \delta'$ contains only empty boxes.
    This means by definition that for any $x\in F$
    \[
        f(x) > t - \delta'
    \]
    or
    \[
        \delta' > t - f(x).
    \]
    Note that $f(x)$ is continuous and $F$ is compact by the problem statement,
    so there exists $x_{min}$ such that
    \[
        \min_F f(x) = f(x_{min}).
    \]
    Combining with the previous inequality we obtain
    \[
        \delta' > t - f(x_{min}) = t - \min_F f(x).
    \]
\end{proof}

\subsection{Proof of theorem~\ref{initial_theorem}}
\begin{proof}
    First statement of the theorem is straightforward.
    Indeed, the initial threshold
    search procedure returns $t$ only if the domain estimation procedure with
    threshold $t$ produces a filled box $B$. The latter means by definition that
    for any point $x\in B$ we have $x\in F$ and
    \begin{equation}
        \label{upper_bound_statement}
        t > f(x) \ge \min_F f(x).
    \end{equation}
    Note that we start with non-empty box and only split boxes into $2^n$ equal
    parts, so $B\ne\emptyset$ and (\ref{upper_bound_statement}) actually holds.

    Let us prove the second statement. Notice that the algorithm runs infinitely
    through step~\ref{it_increase} if it does not stop. So, eventually
    \[
        t > \min_F f(x)
    \]
    becomes true and by problem definition there is a point $x\in F$ such that
    $t > f(x)$. After that the domain estimation procedure always produces
    a boundary or filled box by definition. This forces the initial threshold
    search procedure to either stop or run infinitely through
    step~\ref{it_split}.
    In the latter case the boxes become arbitrarily small and
    the problem approximating domain estimation procedure eventually produce
    a filled box. The initial threshold search algorithm stops after that.
\end{proof}

\subsection{Proof of theorem~\ref{lipschitz_theorem}}
\begin{lemma}
    \label{pattern_continuity}
    Consider a star-shaped polytope $K$ having no faces coplanar with
    its origin $0$. Let $a \ne b$ be points in $\mathbb{R}^n$ such that
    rays $0a$, $0b$ intersect the same simplex face of $K$. Then for any
    point $x$ of the cut $[a,b]$ ray $0x$ intersects $\partial K$ at
    single point belonging to the same face.
\end{lemma}
\begin{proof}
    First notice that any ray starting at $0$ intersects $\partial K$ at
    single point. Indeed, any two points of intersection lie 
    in the same face, otherwise one of them is not visible from $0$
    violating star-shapeness condition. Thus, if there are 
    two distinct intersection points then all the line
    passing through these points and containing $0$ is coplanar
    with the face, which is not allowed by the lemma statement.

    Now denote by $D(y)$ the intersection point of ray $0y$ and $\partial K$.
    It remains to show that for any $x\in[a,b]$ $D(x)$ lies in face
    $F\subset \partial K$ whenever $D(a)\in F$ and $D(b)\in F$. In the rest
    of the proof we interpret all the points as vectors starting at $0$.

    By lemma statement $0\not\in F$, so $D(a)\ne 0$, $D(b)\ne 0$, and
    by definition any $D(y)$ is collinear with $y$.
    Thus, there exist $s>0$, $k>0$ such that
    \[
        D(a) = sa,\,D(b) = kb.
    \]
    Moreover, $D(a)$, $D(b)$ lie in the same hyperplane containing face $F$, so
    there exist $n\in\mathbb{R}^n$, $c\in\mathbb{R}$ such that
    \[
        D(a)\cdot n = c,\,D(b)\cdot n = c.
    \]
    Combining equations we acquire
    \[
        a\cdot n = \frac{c}{s},\, b\cdot n = \frac{c}{k}.
    \]
    Now consider any $x\in [a,b]$. The latter means that there exists
    $\lambda\in[0,1]$ such that
    \[
        x = \lambda a + (1-\lambda) b.
    \]
    Let $D$ be the intersection point of line passing through $0$ and $x$ with
    hyperplane containing face $F$. This means that there exists
    $\mu\in\mathbb{R}\cup\infty$ such that
    \[
        D = \mu x,\, D\cdot n = c.
    \]
    Combining with previous equations we obtain
    \[
        c = \mu x\cdot n = \mu\lambda a\cdot n +\mu(1-\lambda) b\cdot n = 
        \mu\lambda \frac{c}{s} + \mu(1-\lambda) \frac{c}{k}.
    \]
    Notice that by lemma statement $0$ is not coplanar with $F$, so $c\ne 0$.
    Therefore, from the previous equation
    \[
        \mu = \frac{sk}{\lambda k + (1-\lambda)s} \in \mathbb{R}
    \]
    Moreover, $s>0$ and $k>0$, so $\mu>0$, $\mu\ne\infty$ and $D = \mu x$ is
    the intersection
    of ray $0x$ and hyperplane of face $F$. We can express it as follows
    \begin{multline*}
        D = \frac{sk}{\lambda k + (1-\lambda)s}
        \left(\lambda a + (1-\lambda)b\right) = \\
        \frac{sk}{\lambda k + (1-\lambda)s}\left(\frac{\lambda}{s} D(a) + 
        \frac{1-\lambda}{k} D(b)\right) = \\
        \frac{\lambda k}{\lambda k + (1-\lambda)s} D(a) +
        \frac{(1-\lambda)s}{\lambda k + (1-\lambda)s} D(b).
    \end{multline*}
    From the latter we can see that $D$ is a convex combination
    of $D(a)$ and $D(b)$.
    Therefore, $D\in[D(a),D(b)]\subset F$, because by our convention all
    the faces are simplices. As was shown before, every ray starting at
    $0$ intersects $\partial K$ at single point, so $D(x) = D\in F$.
\end{proof}

\begin{lemma}
    \label{lipschitz_lemma}
    Consider a star-shaped polytope $K$ having no faces coplanar with
    its origin $0$. Let $y_1$ and $y_2$ be points in $\mathbb{R}^n$ such that
    rays $0y_1$, $0y_2$ intersect the same simplex face $f$ of $K$.
    Suppose that $r_1>0$ and
    \[
        y_1 \in r_1K
    \]
    then
    \[
        y_2 \in \left(r_1 + L|y_2 - y_1|\right) K
    \]
    where
    \[
        L = \frac{1}{\Delta}
    \]
    and $\Delta$ is the minimum distance from origin $0$ to hyperplanes
    containing pattern faces.
\end{lemma}
\begin{proof}
    Let us put
    \[
        r_2 = r_1 + L|y_2 - y_1|
    \]
    and consider scaled patterns $r_1K$ and $r_2K$.
    Notice that for any $r>0$ homothety with center $0$ and ratio $r$
    transforms $K$ into $rK$ and lefts all the rays starting from $0$ in place.
    Therefore, for any $y\in\mathbb{R}^n$ and $r>0$ ray $0y$ intersects
    the boundary of $rK$
    at image $rf$ of face $f$ whenever $0y$ intersects $\partial K$ at $f$.
    Let us denote by $D_1$, $D_2$ intersection points of rays $0y_1$, $0y_2$
    with faces $r_1f$, $r_2f$ and by $H_1$, $H_2$ hyperplanes
    containing $r_1f$, $r_2f$ respectively, see Fig. \ref{lipschitz_face_figure}.

    \begin{figure}[h]
        \begin{center}
            \includegraphics[width=0.5\textwidth]{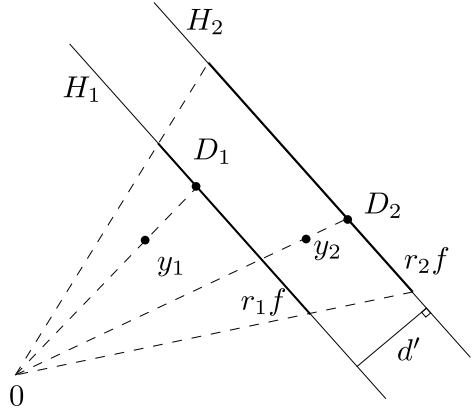}
        \end{center}
        \caption{Proof of lemma~\ref{lipschitz_lemma}}
        \label{lipschitz_face_figure}
    \end{figure}

    Note that $D_1$ and $D_2$ are
    unique due to lemma~\ref{pattern_continuity}.
    It then follows from lemma assumption $y_1\in r_1K$ and star-shapeness of
    $K$ that $|0y_1| \le |0D_1|$, so $y_1$ and $0$ are in the same half-space
    bounded by $H_1$. In addition, notice that
    $r_2 > r_1 > 0$, so it follows from the properties of homothety that
    $H_2$ is further from $0$ and parallel to $H_1$. Consequently, point $0$ and
    $H_2$ lie on the different sides of $H_1$. Combining with the above
    we deduce that $y_1$ and $H_2$ lie on the different sides of $H_1$ as well.
    Thus,
    the distance from $y_1$ to $H_2$ is greater than the distance between
    $H_1$ and $H_2$. Let us estimate the latter.
    Denote by $d$ the
    distance from $0$ to hyperplane containing $f$ in non-scaled pattern $K$. It
    then follows from
    the properties of homothety that the distance between parallel hyperplanes
    $H_1$ and $H_2$ is
    \[
        d' = (r_2 - r_1)d = Ld|y_2-y_1|.
    \]
    By the lemma statement $L = \frac{1}{\Delta}$ and $\Delta \le d$, therefore
    $Ld \ge 1$ and $d' \ge |y_2 - y_1|$. Notice that $|y_2 - y_1|$ is precisely
    the distance between $y_1$ and $y_2$. Now remember that the
    distance between $y_1$ and $H_2$ is greater than $d'$,
    therefore points $y_1$ and $y_2$, as well as $0$, lie in the same half-space
    bounded by $H_2$. Consequently, 
    \[
        |0y_2| \le |0D_2|
    \]
    by definition of $D_2$. It then follows from star-shapeness of $K$
    that
    \[
        y_2 \in r_2K.
    \]
\end{proof}

\begin{proof}[of theorem~\ref{lipschitz_theorem}]
    Consider radius value at any point $x_1\in Q$:
    \[
        r_1 = r(x_1) = \max\{r\ge 0 \mid x_1+rK\subset Q\}
    \]
    and let $x_2\in Q$ be any other point.
    Obviously, there exists a point $D\in (x_1 + r_1K)\cap \partial Q$,
    otherwise $r_1$ can be increased a little bit keeping
    $x_1 + r_1K \subset Q$. In the rest of the proof we treat all the points
    as vectors in $\mathbb{R}^n$ starting at $0$.
    For any point $y\in\mathbb{R}^n$ let us
    denote by $F(y)$ a single pattern face intersecting ray $0y$. Put
    \[
        D_1 = D - x_1,\;D_2 = D - x_2
    \]
    and consider cut $\left[D_1, D_2\right]$.
    From lemma~\ref{pattern_continuity} it follows that any set
    \begin{equation}
        \label{const_face_cut}
        \{y\in \left[D_1,D_2\right] \mid F(y) = f\}
    \end{equation}
    of points mapping to the same face $f$ is a continuous cut inside
    $\left[D_1, D_2\right]$. By definition the pattern boundary consists of
    finite
    number of faces, so $\left[D_1, D_2\right]$ is split into finite number
    of cuts of form (\ref{const_face_cut}), i.e.
    \[
        \left[D_1, D_2\right] = \bigcup_{i=1}^{n-1} \left[y_i, y_{i+1}\right],
    \]
    where the points of every cut $\left[y_i, y_{i+1}\right]$ map to the
    same face under $F$. By definition $D\in x_1 + r_1K$, so
    $D_1 \in r_1K$. Applying
    lemma~\ref{lipschitz_lemma} to every pair of points $y_i$, $y_{i+1}$
    starting with $D_1$ we obtain
    \[
        D_2 \in \left(r_1 + \sum_{i=1}^{n-1} L|y_{i+1} - y_i|\right)K =
        \left(r_1 + L|D_2-D_1|\right)K.
    \]
    The latter combined with definitions of $D_1$ and $D_2$ leads to
    \begin{equation}
        \label{lipschitz_point}
        D = x_2 + D_2 \in x_2 + \left(r_1 + L|x_2 - x_1|\right)K.
    \end{equation}
    Now put
    \[
        r_2 = r_1 + L|x_2 - x_1|
    \]
    and consider any $r > r_2$.
    It then follows from (\ref{lipschitz_point}) and star-shapeness
    of the pattern that
    \begin{equation}
        \label{lipschitz_varpoint}
        D\in x_2 + r_2K \subset x_2 + rK.
    \end{equation}
    From lemma~\ref{pattern_continuity} we obtain that ray $x_2 D$
    intersects $\partial\left(rK\right)$ at single point $D'$, and
    by the properties of homothety
    \[
        D' = x_2 + \frac{r}{r_2}\left(D - x_2\right) \ne D.
    \]
    Combining the latter with (\ref{lipschitz_varpoint}) we deduce that $D$
    is an interior point of $x_2 + rK$. By construction $D\in\partial Q$,
    therefore $x_2 + rK\not\subset Q$ for any $r > r_2$.
    The latter immediately leads to
    \[
        r(x_2) \le r_2 = r_1 + L|x_2 - x_1|
    \]
    which is the desired Lipschitz property.
\end{proof}

\subsection{Proof of theorem~\ref{domain_estimation_theorem}}
\begin{proof}
    First notice that all the filled boxes in the inverse graphical
    scheme are marked precisely by the algorithm of section
    \ref{inverse_impl_section}.
    Indeed, by (\ref{erosion_relation}) at any step of the inverse
    graphical algorithm box $B$ is actually filled if and only if $B$ lies
    inside contour $Q$ and outside $\partial Q + t(-K)$. Moreover,
    Minkowski sum $\partial Q + t(-K)$ is the set union of convex elemets
    \cite{Varadhan2006}, so $B$ lies outside $\partial Q + t(-K)$ if and
    only if $B$ lies outside all the convex elements. So, every actually
    filled box does not intersect any geometry within the algorithm and
    is voxelized precisely as outer for every convex element and inner
    for the contour. This leads to precise marking at step~\ref{vox_unite}
    of the algorithm.

    Now we trivially check definition~\ref{approx_def} for the algorithm
    of this section.
    In scope of problem~\ref{problem} suppose $x\in \sint(F)$ and $f(x) < t$.
    By the continuity of $f$ $f(x) < t$ holds in some neighbourhood
    $U(x)\subset F$.
    For sufficiently small $\delta$ any box $B$ with $\delta(B) <\delta$
    containing $x$ lies inside $U$. By the definition such boxes are filled,
    so they are marked precisely, i.e. $C(B, t) = 1$.
\end{proof}

\subsection{Proof of theorem~\ref{lipschitz_accuracy}}
\begin{proof}
    Note that (\ref{dc_problem}) is a maximization problem,
    so all the relations of
    radius value and threshold through the proof will be in accordance to
    the problem and opposite to the study in section~\ref{gen_section}.

    Now consider any $\varepsilon > 0$ and let us specify values of
    $\delta(\varepsilon)$ and $\theta(\varepsilon)$ to satisfy definition
    \ref{convergent_definition}.
    Indeed, by putting
    \[
        \delta(\varepsilon) = \frac{S}{S + L}\varepsilon
    \]
    we easily acquire
    \[
        \delta' < \left(1 + \frac{L}{S}\right) \delta(\varepsilon) \le
        \varepsilon
    \]
    for the accuracy estimate $\delta'$
    whenever the curent accuracy satisfies $\delta < \delta(\varepsilon)$.
    The value of $\theta(\varepsilon)$ have no impact on the further
    study, so we just put $\theta(\varepsilon) = \varepsilon$.
    It then remains
    to make sure that $\delta'$ is the actual accuracy estimate, i.e.
    \[
        \delta' \ge \max_Q r(x) - t
    \]
    holds at every step~\ref{alg_test_comp} of the inverse graphical
    optimization algorithm. Indeed, the move from step~\ref{alg_bound}
    to step~\ref{alg_test} in the optimization algorithm happens
    only when there are no boxes marked
    as filled by the domain estimation procedure
    for threshold $t + \delta$. In addition, it have been shown 
    during the proof of theorem \ref{domain_estimation_theorem} that
    all the actually filled boxes are marked as filled.
    This means that every remaining active box
    $X_i\in\mathcal{X}$ is either outside the contour or contains
    a point $x_i$ within $\partial Q + (t + \delta)(-K)$ meaning that
    \[
        r(x_i) \le t + \delta.
    \]
    By theorem~\ref{lipschitz_theorem} for every $x\in X_i\cap Q$
    \[
        r(x) - r(x_i) \le L\left|x-x_i\right| \le L\delta(X_i),
    \]
    where $\delta(X_i)$ is the diameter of active box $X_i$.
    Therefore,
    \[
        r(x) \le t + \delta + L\delta(X_i)
    \]
    for every $x\in Q$ within active boxes.
    By step~\ref{alg_iterate} of the algorithm $\delta = S\delta(X_i)$,
    hence
    \begin{equation}
        \label{tp_estimate}
        r(x) \le t + \delta + \frac{L}{S}\delta = t + \delta'.
    \end{equation}
    This holds for every feasible point in every active box. For all the prunned
    boxes $r(x) \le t$ by step~\ref{alg_bound} of the algorithm. Therefore,
    (\ref{tp_estimate}) holds globally and
    \[
        \max_Q r(x) \le t + \delta'
    \]
    or equivalently
    \[
        \delta' \ge \max_Q r(x) - t.
    \]
    This is exactly the definition of the actual accuracy estimate holding at
    every step~\ref{alg_test_comp} of the algorithm.
\end{proof}

\section*{Acknowledgements}
The author wishes to thank Vladimir Valedinsky(Lomonosov Moscow State University)
for his help at the early stages of the work.


\end{document}